\documentclass[11pt,bezier]{article}
\usepackage{amsmath,amssymb,amsfonts,euscript,graphicx}

\usepackage[margin=.75in]{geometry}

\title{{\bf The Annihilating-Ideal Graph of a Ring}}
\author{{{\bf ${{\rm { {\bf F. ~Aliniaeifard }}}}$\thanks
{ Corresponding author: Farid Aliniaeifard, Department of Mathematics, Brock University, St. Catharines, Ontario,
Canada L2S 3A1. E-mail address: fa11da@brocku.ca},~ ${{\rm {\bf M.~ Behboodi}}}$}}\thanks
{ The research of the second author was  supported in part by a grant from IPM (No. 90160034).} {\rm{\bf~ and}}
${{\rm {\bf Y.~ Li}}}$\thanks
{ The research of the first and the third authors was supported in part by a Discovery Grant from the Natural Sciences and Engineering Research
Council of Canada.}
\thanks{The research of third author was also supported in part by the National Natural Science Foundation of China (No. 11271250).}\\
\\
\footnotesize{${^{\rm }}$}\vspace{-1mm}
}

\def\be{\begin{enumerate}}
\def\ee{\end{enumerate}}

\newtheorem{ttheo}{Theorem}[section]
\newtheorem{ccoro}[ttheo]{Corollary}
\newtheorem{llem}[ttheo]{Lemma}
\newtheorem{example}[ttheo]{Example}
\newtheorem{rrem}[ttheo]{Remark}
\newtheorem{ppro}[ttheo]{Proposition}

\newenvironment{pproof}{\noindent{\bf Proof. }}{}

\date{}
 \begin{document}
  \maketitle
\begin{abstract}
{Let $S$ be a semigroup with $0$ and $R$ be a ring with $1$. We extend the definition of the zero-divisor graphs of commutative semigroups to not necessarily commutative semigroups. We define an annihilating-ideal graph of a ring as a special type of zero-divisor graph of a semigroup. We introduce two
ways to define the zero-divisor graphs of  semigroups. The first definition gives a directed graph ${\Gamma}(S)$, and the other definition yields an undirected graph $\overline{{\Gamma}}(S)$. It is shown that $\Gamma(S)$ is not necessarily connected, but $\overline{{\Gamma}}(S)$ is  always connected and ${\rm diam}(\overline{\Gamma}(S))\leq 3$. For a ring $R$  define a directed graph $\Bbb{APOG}(R)$ to be equal to  $\Gamma(\Bbb{IPO}(R))$, where  $\Bbb{IPO}(R)$ is a semigroup consisting of all products of two one-sided ideals of $R$, and define an undirected graph $\overline{\Bbb{APOG}}(R)$ to be equal to  $\overline{\Gamma}(\Bbb{IPO}(R))$. We show that $R$ is an Artinian (resp., Noetherian) ring if and only if $\Bbb{APOG}(R)$ has DCC (resp., ACC) on some special subset of its vertices. Also, It is shown that  $\overline{\Bbb{APOG}}(R)$ is a complete graph if and only if either $(D(R))^{2}=0$, $R$ is a direct product of two division rings, or $R$ is a local ring with maximal ideal $\mathfrak{m}$ such that $\Bbb{IPO}(R)=\{0,\mathfrak{m},\mathfrak{m}^{2}, R\}$. Finally, we investigate the diameter and the girth of square matrix rings over commutative rings $M_{n\times n}(R)$ where $n\geq 2$.  
  }
\end{abstract}


\section{introduction}

In \cite{Be88}, I. Beck associated to a commutative ring $R$ its zero-divisor graph
$G(R)$ whose vertices are the zero-divisors of $R$ (including $0$), and two distinct vertices
$a$ and  $b$ are adjacent if $ ab = 0$. In \cite{AL99}, Anderson and Livingston introduced and studied the subgraph
$\Gamma(R)$ (of $G(R)$) whose vertices are the nonzero zero-divisors of $R$. This graph turns out
to best exhibit the properties of the set of zero-divisors of $R$, and the ideas and
problems introduced in \cite{AL99} were further studied in \cite{AM04, AB13, ALS03}. In \cite{Re01}, Redmond extended the definition of zero-divisor graph to non-commutative rings. Some fundamental results concerning zero-divisor graph for a
non-commutative ring were given in \cite{AM041,AM042,Wu05}.
For a commutative ring $R$ with $1$, denoted by $\Bbb{A}(R)$, the set of ideals with nonzero annihilator. The annihilating-ideal graph of
$R$ is an undirected graph $\Bbb{AG}(R)$ with vertices $\Bbb{A}(R)^{*} = \Bbb{A}(R)
\setminus \{0\}$, where distinct vertices $I$ and $J$ are
adjacent if $IJ = (0)$. The concept of the
annihilating-ideal graph of a commutative ring was introduced in \cite{BRI11,BRII11}. Several fundamental results concerning $\Bbb{AG}(R)$ for a
commutative ring were given in \cite{AAB,AANS,AAN13,AB12}. For a ring $R$, let $D(R)$ be the set of one-sided zero-divisors of $R$ and $\Bbb{IPO}(R) = \{{\rm {\it A \subseteq R : A = IJ} ~where~ {\it I} ~and ~{\it J} ~are~ left~ or~  right~ ideals~ of}~R\}$. Let $S$ be a semigroup with $0$, and $D(S)$ be the set of one-sided zero-divisors of $S$.
 The zero-divisor graph of a commutative semigroup is an undirected graph with vertices $Z(S)^{*}$ (the set of non-zero zero-divisors) and two distinct vertices $a$ and $b$ are adjacent if $ab=0$. The zero-divisor graph of a commutative semigroup was introduced in \cite{DMS02} and further studied in \cite{DD05,WL07,WL06,WQL09}.
\\

Let $\Gamma$ be a graph. For vertices $x$ and $y$ of  $\Gamma$, let $d(x, y)$ be the length of a shortest path
from $x$ to $y$ ($d(x, x) = 0$ and $d(x, y) = \infty$ if there is
no such a path). The diameter of $\Gamma$ is defined as $ {\rm diam}(\Gamma) = sup\{d(x, y) |$ $x$
and $y$ are vertices of $\Gamma\}$.
The girth of $\Gamma$, denoted by
${\rm gr}(\Gamma)$, is the length of a shortest cycle in $\Gamma$ (${\rm gr}(\Gamma) =
\infty$ if $\Gamma$ contains no cycles).
\\

In Section 2, we introduce a directed graph $\Gamma(S)$ for a semigroup $S$ with 0. We show that  $\Gamma(S)$ is not necessarily connected. Then we find a necessarily and sufficient condition for $\Gamma(S)$ to be connected. After that  we extend the annihilating-ideal graph to a (not necessarily commutative) ring. It is shown that $\Bbb{IPO}(R)$ is a semigroup. We associate to a  ring $R$ a directed graph (denote by $\Bbb{APOG}(R)$) the zero-divisor graph of $\Bbb{IPO}(R)$, i.e.,  $\Bbb{APOG}(R)=\Gamma(\Bbb{IPO}(R))$. Then we show that $R$ is an Artinian (resp., Noetherian) ring if and only if $\Bbb{APOG}(R)$ has DCC (resp., ACC) on some subset of its vertices.
In Section 3, we introduce an undirected graph $\overline{\Gamma}(S)$ for a semigroup $S$ with 0. We show that  $\overline{\Gamma}(S)$ is always connected and ${\rm diam}(\overline{\Gamma}(S))\leq 3$. Moreover, if $\overline{\Gamma}(S)$ contains a cycle, then ${\rm gr}(\overline{\Gamma}(S))\leq 4$. After that  we define an undirected graph which extends the annihilating-ideal graph to a not necessarily commutative ring. We associate to a  ring $R$ an undirected graph (denoted by $\overline{\Bbb{APOG}}(R)$) the undirected zero-divisor graph of $\Bbb{IPO}(R)$, i.e.,  $\overline{\Bbb{APOG}}(R)=\overline{\Gamma}(\Bbb{IPO}(R))$. Finally, we characterize rings whose undirected annihilating-ideal graphs are complete graphs.
In Section 4, we investigate the undirected annihilating-ideal graphs of matrix rings over  commutative rings. It is shown that ${\rm diam}((\overline{\Bbb{APOG}}(M_{n}(R)))\geq 2$ where $n\geq 2$. Also, we show that  ${\rm diam}(\overline{\Bbb{APOG}}(M_{n}(R))\geq {\rm diam}(\overline{\Bbb{APOG}}(R))$.



\section{Directed Annihilating-Ideal Graph of a  Ring}

Let $S$ be a semigroup with $0$ and $D(S)$ denote the set of one-sided zero-divisors of $S$. We associate to $S$ a directed graph $\Gamma(S)$ with vertices set $D(S)^{*}=D(S)\setminus \{0\}$ and $a\rightarrow b$ if $ab=0$. In this section, we investigate the properties of $\Gamma(S)$ and we first show the following result.

\begin{ppro}
Let $R$ be a ring. Then $\Bbb{IPO}(R)$ is a semigroup.
\end{ppro}

\begin{pproof}
 Let $A,B \in \Bbb{IPO}(R)$. Then
there exist left or right ideals $I_{1}, J_{1}, I_{2}, J_{2}$ of $R$ such that
$A=I_{1}J_{1}$ and $B=I_{2}J_{2}$. We show that
$AB=(I_{1}J_{1})(I_{2}J_{2})\in \Bbb{IPO}(R)$.

$Case$ 1:  $J_{1}$ is a left
ideal. Then $AB=I_{1}(J_{1}I_{2}J_{2})
\in \Bbb{IPO}(R)$ (as $J_{1}I_{2}J_{2}$ is a left ideal of $R$).

$Case$ 2: $J_{1}$ is a right ideal and either $I_{2}$ is a left ideal or
$J_{2}$ is a right ideal. Then $AB=(I_{1}J_{1})(I_{2}J_{2})\in
\Bbb{IPO}(R)$.

$Case$ 3: $J_{1}$ is a right ideal, $I_{2}$ is a right ideal, and
$J_{2}$ is a left ideal. Then $AB=(I_{1}J_{1}I_{2})J_{2}\in
\Bbb{IPO}(R)$.

Thus $\Bbb{IPO}(R)$ is closed multiplicatively. Since the multiplication is associative, $\Bbb{IPO}(R)$ is a semigroup.\hfill $\square$
\end{pproof}
\\

 It was shown in \cite[Theorem 1.2]{DMS02} that the zero-divisor graph of a commutative semigroup $S$ is connected and ${\rm diam} (\Gamma(S))\leq 3$ . In the following example we show that $\Gamma(S)$  is not necessarily  connected when $S$ is a non-commutative semigroup.

\begin{example} Let $K$ be a field and $V =
\oplus_{i=1}^{\infty}K$. Then $R = HOM_{K}(V,V)$, under
the point-wise addition and the multiplication taken to be the composition of
functions, is an infinite non-commutative ring with identity.
Let $\pi_{1}: V \rightarrow V$ be defined by $(a_{1},a_{2}, ...)
\mapsto  (a_{1},0, ...)$ and $f : V \rightarrow V$ be defined by
$(a_{1},a_{2}, ...) \mapsto (0, a_{1},a_{2}, ...)$. Then $\pi_{1}, f \in R$. Note that $(R\pi_{1})(fR) = 0 $, so $\Gamma(\Bbb{IPO}(R)) \neq\emptyset$. However, $\Gamma(\Bbb{IPO}(R))$ is not connected as there is no
path leading from the vertex $(fR)$ to any other vertex of
 $\Gamma(\Bbb{IPO}(R))$.  This is because there exists $g : V \rightarrow V$ given by
$(a_{1},a_{2}, ...) \mapsto (a_{2},a_{3}, ...)$ and $g \in R$
such that $gf = 1_{R}$.\hfill $\square$
\end{example}

For a semigroup $S$, let  $$A^{l}(S) =\{a \in D(S)^{*}:~{\rm  there ~exists}~ b \in D(R)^{*}
{\rm ~such~ that~} ba = 0\}$$ and
$$A^{r}(S) = \{a \in D(S)^{*}:~ {\rm there ~exists}~ b \in D(R)^{*}~ {\rm such~ that} ~ab = 0 \}.$$
Next we show that $\Gamma(S)$ is connected if and only if $A^{l}(S)=A^{r}(S)$. Moreover, if $\Gamma(S)$ is connected, then ${\rm diam}(\Gamma(S))\leq 3$.

\begin{ttheo}\label{diam33}
Let $S$ be a  semigroup. Then $\Gamma(S)$ is connected if
and only if $A^{l}(S) = A^{r}(S)$. Moreover, if $\Gamma(S)$ is
connected, then $diam(\Gamma(S)) \leq 3$.
\end{ttheo}

\begin{pproof}
Suppose that $A^{l}(S) = A^{r}(S)$.\\
Let $a$ and $b$ be distinct vertices of $\Gamma(S)$. Then $a \neq
0$ and $b \neq 0$. We show that there is always a path with length at most 3 from $a$ to $b$.

$Case ~1$: $ab = 0$. Then $a \rightarrow b$ is a desired path.

$Case ~2$:  $ab \neq 0$. Then since $A^{l}(S) = A^{r}(S)$,
 there exists $c \in D(S) \setminus \{0\}$ such that
 $ac = 0$ and $d \in D(S) \setminus \{0\}$ such
 that $db= 0$.

$Subcase~ 2.1$: $c= d$. Then $a \rightarrow c \rightarrow b$ is a desired
 path.

$Subcase ~2.2$: $c \neq d$. If $cd = 0$, then $a \rightarrow c \rightarrow d \rightarrow b$ is a desired
 path. If $cd \neq 0$, then $a \rightarrow cd \rightarrow b$ is
 a desired path.

Thus $\Gamma(S)$ is connected and ${\rm diam}(\Gamma(S)) \leq 3$.

Conversely, if $\Gamma(S)$ is connected, then it is easy to show that $A^{l}(S) = A^{r}(S)$.\hfill $\square$
 \end{pproof}
\\

Now, we define a directed graph which extends the annihilating-ideal graph to an arbitrary ring. We associate to a  ring $R$ a directed graph (denoted by $\Bbb{APOG}(R)$) the zero-divisor graph of $\Bbb{IPO}(R)$, i.e.,  $\Bbb{APOG}(R)=\Gamma(\Bbb{IPO}(R))$.

\begin{ccoro}
Let $R$ be a ring. Then $\Bbb{APOG}(R)$ is connected if and only if $A^{l}(\Bbb{IPO}(R))=A^{r}(\Bbb{IPO}(R))$. Moreover, if $\Bbb{APOG}(R)$ is connected, then ${\rm diam}(\Bbb{APOG}(R))\leq 3$.
\end{ccoro}

\begin{pproof}
Since $\Bbb{APOG}(R)$ is equal to $\Gamma(\Bbb{IPO}(R))$,  it follows from Theorem \ref{diam33} that $\Bbb{APOG}(R)$ is a connected if and only if $A^{l}(\Bbb{IPO}(R))=A^{r}(\Bbb{IPO}(R))$. Also,  if $\Bbb{APOG}(R)$ is connected, then ${\rm diam}(\Bbb{APOG}(R))\leq 3$.
\end{pproof}
\\

Recall that a Duo ring is a ring in which every one-sided ideal is a two-sided ideal.

\begin{ppro}\label{Duo}
 Let $R$ be an Artinian Duo ring.  Then $A^{l}(\Bbb{IPO}(R))=A^{r}(\Bbb{IPO}(R)) = \Bbb{IPO}(R)\setminus \{0, R\}$. Moreover, $\Bbb{APOG}(R)$ is connected and ${\rm diam}(\Bbb{APOG}(R))\leq 3$.
\end{ppro}

\begin{pproof}
Let $R$ be a Duo ring. Then by \cite[Lemma 4.2]{KK10}, $R= (R_{1},\mathfrak{m}_{1}) \times (R_{2},\mathfrak{m}_{2}) \times \cdots (R_{n},\mathfrak{m}_{n})$, where each $R_{i} (1\leq i\leq n)$ is an Artinian local ring with unique maximal ideal $\mathfrak{m}_{i}$. Let $A \in \Bbb{IPO}(R)\setminus \{0, R\}$. Then $A = (I_{1} \times I_{2} \times \cdots \times I_{n})$
 $(J_{1} \times J_{2} \times ... \times J_{n})$, where every $I_{i} (1\leq i\leq n$) is an one-sided ideal, so is every $J_{j}(1\leq j \leq n)$. Since $A \neq R$, there exists
 $I_{i}$ (or $J_{j}$) such that $I_{i}\neq R$ (or $J_{j} \neq R$). Without loss of generality we may assume that $I_{i} \neq R$. So $A = (I_{1} \times I_{2} \times \cdots \times I_{n})$ $(J_{1} \times J_{2} \times \cdots \times J_{n})$
 $ \subseteq $ $(R_{1} \times \cdots\times I_{i} \times \cdots \times R_{n})$  $(R_{1} \times \cdots \times R_{i} \times \cdots \times
 R_{n})$. Suppose $k$ is the smallest positive integer such that ${I_{i}}^{k} =
 0 $. Thus $(0 \times \cdots \times I_{i}^{k - 1} \times ... \times
 0) ((R_{1} \times \cdots \times I_{i} \times \cdots \times R_{n})(R_{1} \times \cdots \times R_{i} \times \cdots \times
 R_{n})) = 0 $ and $((R_{1} \times \cdots \times I_{i} \times \cdots \times R_{n})(R_{1} \times \cdots \times R_{i} \times \cdots \times
 R_{n}))(0 \times \cdots \times I_{i}^{k - 1} \times \cdots \times
 0) = 0$. Therefore $A \in A^{l}(\Bbb{IPO}(R))$ and  $A\in A^{r}(\Bbb{IPO}(R))$. Thus $\Bbb{IPO}(R)\setminus \{0, R\}\subseteq A^{r}(\Bbb{IPO}(R))$ and $\Bbb{IPO}(R)\setminus \{0, R\}\subseteq A^{l}(\Bbb{IPO}(R))$. We conclude that $A^{r}(\Bbb{IPO}(R))=\Bbb{IPO}(R)\setminus \{0, R\}=A^{l}(\Bbb{IPO}(R))$.

The second part follows from Theorem \ref{diam33}. \hfill $\square$
 \end{pproof}
\\

 It is well known that if $|D(R)|\geq 2$ is finite, then $|R|$
  is finite. Let $A,B$ be vertices of $\Bbb{APOG}(R)$. We use $A\rightleftharpoons B$ if $A\rightarrow B$ or $A\leftarrow B$.  For any vertices $C$ and $D$ of $\Bbb{APOG}(R)$, let ${\rm ad}(C)$ = $\{A$ is a vertex of  $\Bbb{APOG}(R)$ :  $C = A$ or $C \rightleftharpoons A$ or there exists a vertex $B$ of  $\Bbb{APOG}(R)$ such that $C \rightleftharpoons  B \rightleftharpoons A$ $\}$ and ${\rm adu}(D)=\bigcup_{C\subseteq D}{\rm ad}(C)$. We know that $ {\rm ad}(C)\subseteq D(R)$. The following proposition shows that if a principal left or right ideal $I$ of $R$ is a vertex of $\Bbb{APOG}(R)$  and all  left and right ideals of ${\rm ad}(I)$ have finite cardinality, then $R$ has finite cardinality.

 \begin{ppro}
 Let $R$ be a ring and $I$ be a
 principal left or right ideal of $R$ such that $I$ is a vertex of $\Bbb{APOG}(R)$. If all left and right ideals of ${\rm ad}(I)$ have finite cardinality, then $R$ has finite cardinality.
 \end{ppro}

\begin{pproof}
Without loss of generality, we may assume that $I$ is a left principal ideal. Thus
$I = Rx$ for some non-zero $x \in R$. If $Ann_{l}(x) = 0$, then $| R | = | I| <\infty$. So we may always assume that $Ann_{l}(x) \not= 0$.

$ Case ~1$: $I = Ann_{r}(x)$ and $Ann_{r}(x)Ann_{l}(x) = 0$.  Then
 $$I \rightarrow Ann_{l}(x)$$ and so $Ann_{l}(x) \in {\rm ad}(I)$.
 Therefore, $Ann_{l}(x)$ is finite. Since $I \cong R/
 Ann_{l}(x)$, $| R | = | I| | Ann_{l}(x)|<\infty$.

$ Case ~2$: $I \neq Ann_{r}(x)$ and $Ann_{r}(x)Ann_{l}(x) = 0$. If $Ann_{r}(x) \not= 0$,
 then $$I \rightarrow Ann_{r}(x) \rightarrow Ann_{l}(x)$$ and so $Ann_{l}(x) \in {\rm ad}(I)$.
 Therefore, $Ann_{l}(x)$ is finite. Since $I \cong R/
 Ann_{l}(x)$, $| R | = | I| | Ann_{l}(x)|<\infty$. If $Ann_{r}(x) = 0$, then since  $Rx$ is a vertex of $\Bbb{APOG}(R)$, there exists a  (nonzero right ideal) $J$ such that $JRx=0$ (replace $J$ by $JR$ if necessary). Since $Ann_{r}(x) = 0$, we have $xJ$ is a nonzero right ideal and so $$ Ann_{l}(x) \rightarrow xJ \rightarrow I.$$ Thus $Ann_{l}(x) \in {\rm ad}(I)$, so $Ann_{l}(x)$ is finite. Again, we have $| R | = | I| | Ann_{l}(x)|<\infty$.

$ Case ~3$:  $I \neq Ann_{r}(x)$ and $ Ann_{r}(x)Ann_{l}(x) \neq 0 $.
 Then $$Ann_{r}(x)\leftarrow I \rightarrow Ann_{r}(x)Ann_{l}(x) \rightarrow (xR)$$
 and so $(xR), Ann_{r}(x) \in {\rm ad}(I)$. Therefore,  $(xR)$ and  $Ann_{r}(x)$
 are finite. Since $(xR) \cong R/
 Ann_{r}(x)$, $| R | = | (xR)| | Ann_{r}(x)|<\infty$. This completes the proof. \hfill $\square$
 \end{pproof}\\

Here is our main result in this section.

\begin{ttheo}\label{Artin}
Let $R$ be a ring such that $\Bbb{APOG}(R)\neq \emptyset$. Then $R$ is Artinian (resp., Noetherian) if and only if for a left or right ideal $I$ in the vertex set of  $\Bbb{APOG}(R)$, ${\rm adu}(I)$ has DCC (resp., ACC) on both its left and right ideals.
\end{ttheo}

\begin{proof}
If $R$ is Artinian, then $\Bbb{IPO}(R)$ has DCC  on  both its left ideals and right ideals. Thus for every left or right ideal of the vertex set of  $\Bbb{APOG}(R)$, ${\rm adu}(I)$ has DCC on both its left and right ideals  as ${\rm adu}(I)\subseteq \Bbb{IPO}(R)$.

Conversely, without loss of generality let $I$ be a left ideal of vertex set of $\Bbb{APOG}(R)$ such that ${\rm adu}(I)$ has DCC on its left and right ideals. Assume that $x\in I$.  We have the following cases:

$Case~1$: $xRx\neq \{0\}$, $Ann_l(x)\neq 0$, and $Ann_r(x)\neq 0$. Then $$(xR)\leftarrow Ann_l(x)\leftarrow xRx\rightarrow Ann_{r}(x)\leftarrow (Rx).$$
Therefore $(xR),Ann_r(x), Ann_l(x), (Rx)\in {\rm ad}(xRx)$. Since ${\rm ad}(xRx)\subseteq {\rm adu}(I)$ and ${\rm adu}(I)$ has DCC  on its left and right ideals, we conclude that $(Rx)$ and $Ann_l(x)$ are left Artinian $R$-modules, and $(xR)$ and $Ann_r(x)$ are right Artinian  $R$-modules. Since $(Rx)\cong R/Ann_l(x)$ and $(xR)\cong R/Ann_r(x)$,   by \cite[(1.20)]{La91} we conclude that  $R$ is Artinian.

$Case~2$: $xRx=\{0\}$, $Ann_l(x)\neq 0$, and $Ann_r(x)\neq 0$.  Then $$Ann_l(x)\rightarrow(xR)\rightarrow(Rx)\rightarrow Ann_r(x).$$ Since ${\rm ad}(Rx)\subseteq {\rm adu}(I)$  and ${\rm adu}(I)$ has DCC  on its  left and right ideals, we conclude that $(Rx)$ and $Ann_l(x)$ are left Artinian  $R$-modules, and $(xR)$ and $Ann_r(x)$ are right Artinian $R$-modules. Since $(Rx)\cong R/Ann_l(x)$ and $(xR)\cong R/Ann_r(x)$,  by \cite[(1.20)]{La91} we conclude that $R$ is Artinian.

$Case~3$: $Ann_l(x)=\{0\}$. Then $Rx\cong R$. Therefore, $R$ is a left Artinian module. Since $Rx$ is a vertex of  $\Bbb{APOG}(R)$, we have $Ann_r(x)\neq \{0\}$. So there exists $y\in D(R)\setminus \{0\}$ such that $xy=0$.

$Subcase~3.1$: $yRy\neq \{0\}$.  If $Ann_r(y)=\{0\}$, then since $$Rx\rightarrow yR,$$ we have $yR\in {\rm adu}(I)$, so $yR$ is a Artinian right $R$-module. Note that $yR\cong R$. Therefore, $R$ is a right Artinian module. If  $Ann_r(y)\neq\{0\}$, then $$Ann_r(y)\leftarrow yRy \leftarrow yRx \rightarrow yR.$$ Therefore $(yR),Ann_{r}(y)\in {\rm ad }(yRx)\subseteq {\rm adu}(I)$.  Since ${\rm adu}(I)$ has DCC  on its right ideals, we conclude that $(yR)$ and $Ann_r(y)$ are right Artinian $R$-modules. Note that $(yR)\cong R/Ann_r(y)$, by \cite[(1.20)]{La91} we conclude that $R$ is a right Artinian module.

$Subcase~3.2$: $yRy= \{0\}$. Then $$yR\leftarrow yRx\leftarrow Ry\rightarrow Ann_r(y).$$ Since $(yR), Ann_r(y)\in {\rm ad}(yRx)\subseteq {\rm adu}(I)$,  we conclude that $(yR)$ and $Ann_r(y)$ are right Artinian $R$-modules. Note that $(yR)\cong R/Ann_r(y)$, by \cite[(1.20)]{La91}  we conclude that $R$ is a right Artinian module.

$Case~4$: $Ann_r(x)=\{0\}$. Then $xRx\neq \{0\}$ and since $Rx$ is a vertex of $\Bbb{APOG}(R)$, we have $Ann_l(x)\neq \{0\}$. Therefore, $$(xR)\leftarrow Ann_l(x)\rightarrow xRx.$$ We conclude that $xR,Ann_l(x)\in {\rm ad}(xRx)\subseteq {\rm adu}(I)$. Since  $xR,Rx,Ann_l(x)\in {\rm adu}(I)$, we have $Rx$ and $Ann_l(x)$ are left Artinian modules and $xR$ is a right Artinian module.   Note that  $(Rx)\cong R/Ann_l(x)$ and $(xR)\cong R/Ann_r(x)$. Again by \cite[(1.20)]{La91} we conclude that $R$ is Artinian.\hfill $\square$
\end{proof}
\\

\begin{ccoro}\label{Artin2}
Let $R$ be a ring such that $\Bbb{APOG}(R)\neq \emptyset$. Then $R$ is Artinian (resp., Noetherian) if and only if  $\Bbb{APOG}(R)$ has DCC (resp., ACC) on left and right ideals of its vertex set.
\end{ccoro}

\begin{pproof}
Since vertex set of  $\Bbb{APOG}(R)$ is a subset of $\Bbb{IPO}(R)$, As in the proof of Theorem \ref{Artin},  if $R$ is Artinian (resp., Noetherian), then  $\Bbb{APOG}(R)$ has DCC (resp., ACC) on left and right ideals of its vertex set. \\

Conversely, since for a left or right ideal $I$ of the vertex set of $\Bbb{APOG}(R)$, ${\rm adu}(I)$ is a subset of the vertex set of $\Bbb{APOG}(R)$, it follows from Theorem \ref{Artin} that $R$ is Artinian.\hfill $\square$
\end{pproof}
\\

A directed graph $\Gamma$ is called a
tournament if for every two distinct vertices $x$ and $y$ of $\Gamma$
exactly one of $xy$ and $yx$ is an edge of $\Gamma$. In other words,  a
tournament is a complete graph with exactly one direction assigned
to each edge.

\begin{ppro} Let $R$ be a ring such that
$A^{2}\neq \{0\}$ for every non-zero $A \in \Bbb{IPO}(R)$ and
$A^{l}(\Bbb{IPO}(R))\cap A^{r}(\Bbb{IPO}(R)) \neq \emptyset$. Then $\Bbb{APOG}(R)$ is not a tournament.
\end{ppro}

\begin{pproof}
Assume $\Bbb{APOG}(R)$ is a tournament. Since
$A^{l}(\Bbb{IPO}(R)) \cap A^{r}(\Bbb{IPO}(R)) \neq \emptyset$, there exists $B \in
A^{l}(\Bbb{IPO}(R))\cap A^{r}(\Bbb{IPO}(R))$, that is, there exist distinct non-zero
$A, C \in \Bbb{IPO}(R)$ such that $A\rightarrow B\rightarrow C$
is a path in $\Bbb{APOG}(R)$. If $CA\neq \{0\}$, then $B(CA)=(BC)A=\{0\}$ and
$(CA)B=C(AB)=\{0\}$, which is a contradiction. So $CA=\{0\}$ and
therefore $AC\neq \{0\}$ since $\Bbb{APOG}(R)$ is a tournament. Also, $AC
\neq A$ (otherwise $A^{2}= (ACAC)=A(CA)C = \{0\}$) and similarly,
$AC\neq C$. Let $a,a_1\in A$ and $c,c_1\in C$. Then we have $B\rightarrow
C\rightarrow ((a-a_1c)R)$ and $(R(c-ac_1))\rightarrow A\rightarrow B$. As the above $((a-a_1c)R)B=\{0\}$ and $B(R(c-ac_1))=\{0\}$. Let $b\in B$ be an arbitrary element. Then $-acb=a_1b-acb\in ((a-a_1c)R)B=\{0\}$ and $bac=bc_1-bac\in B(R(c-ac_1))=\{0\}$. Therefore, $ACB=\{0\}$ and $BAC=\{0\}$. Thus both $AC\rightarrow B$ and $B\rightarrow AC$ are edges of
$\Bbb{APOG}(R)$. This is a contradiction, hence, $\Bbb{APOG}(R)$ cannot be a
tournament.\hfill $\square$

\end{pproof}


\section{Undirected Annihilating-Ideal Graph of a Ring}

Let $S$ be a semigroup with $0$ and recall that $D(S)$ denotes the set of one-sided zero-divisors of $S$. We associate to $S$ an undirected graph $\overline{\Gamma}(S)$ with vertices set $D(S)^{*}=D(S)\setminus \{0\}$ and two distinct vertices $a$ and $b$ are adjacent if $ab=0$ or $ba=0$.
 Similarly, we associate to a  ring $R$ an undirected graph (denoted by $\overline{\Bbb{APOG}}(R)$) the undirected zero-divisor graph of $\Bbb{IPO}(R)$, i.e.,  $\overline{\Bbb{APOG}}(R)=\overline{\Gamma}(\Bbb{IPO}(R))$.
The only difference between $\Bbb{APOG}(R)$ and $\overline{\Bbb{APOG}}(R)$
is that the former is a directed graph and the latter is
undirected (that is, these graphs share the same vertices and the
same edges if directions on the edges are ignored). If $R$ is a
commutative ring, this definition agrees with the previous
definition of the annihilating-ideal graph. In this section we study the properties of $\overline{\Gamma}(R)$.  We first show that $\overline{\Gamma}(R)$ is always connected with diameter at most 3.

\begin{ttheo}\label{diam}
 Let $S$ be a semigroup. Then
$\overline{\Gamma}(S)$ is a connected graph and
${\rm diam}(\overline{\Gamma}(S)) \leq 3.$
\end{ttheo}

\begin{pproof}
Let $a$ and $b$ be distinct vertices of
$\overline{\Gamma}(S)$. If $ab = 0$ or $ba =  0$, then $a - b$ is a path. Next assume that $ab \neq  0$ and $ba \neq  0$.

$Case ~1$: $a^{2} =  0$ and $b^{2} =  0$. Then $a - ab - b$ is
a path.

$Case ~2$: $a^{2} =  0 $ and $b^{2} \neq  0$. Then there is a
some $c \in D(S) \setminus \{a, b , 0\}$ such that either
$cb =  0$ or $bc =  0$. If either $ac =  0$ or $ca =  0$, then
$a - c - b$ is a path. If $ac \neq  0$ and $ca\neq  0$, then $a
- ca - b$ is a path if $bc =
 0$ and $ a - ac - b$ is a path if $cb =  0$.

$Case ~3$: $a^{2} \neq  0$ and $b^{2} =  0$. We can use an
argument
similar to that of the above case to obtain a path.

$Case ~4$: $a^{2} \neq  0$ and $b^{2} \neq  0$. Then there
exist $c, d \in D(S) \setminus \{a, b , 0\}$ such that
either $ca =  0$ or $ac =  0$ and either $db = 0$ or $bd =
 0$. If $bc = 0$ or $cb=0$, then $a - c - b$ is a path. Similarly, if $ad=0$ or $ da=0$,  $a - d - b$ is a path. So we may assume that $c\neq d$. If $cd=0$ or $dc=0$, then $a - c - d - b$ is a path. Thus we may further assume that  $cd \neq  0, dc \neq  0$, $bc \neq  0, cb \neq  0 $, $ad \neq 0$ and $da \neq 0.$  We divide the proof into 4 subcases.

$Subcase ~4.1$: $ac =  0$ and $db =  0$. Then $a - cd - b$ is a path.

$Subcase ~4.2$: $ac =  0$ and  $bd =  0$. Then $a - cb - d - b$ is a path.

$Subcase ~4.3$: $ca =  0$ and $bd =  0$. Then $a - dc - b$ is a path.

$Subcase ~4.4$: $ca =  0$ and $db = 0$.  $a - bc - d - b$ is a path.

Thus $\overline{\Gamma}(S)) $ is connected and
${\rm diam}(\overline{\Gamma}(S)) \leq 3.$ \hfill $\square$
\end{pproof}\\

In \cite {AL99}, Anderson and Livingston proved that if $\Gamma(R)$ (the zero-divisor graph of a commutative ring $R$) contains a cycle, then ${\rm gr} (\Gamma(R))\leq 7$. They also proved that ${\rm gr}(\Gamma(R)) \leq 4$  when $R$ is Artinian and
conjectured that this is the case for all commutative rings $R$. Their conjecture was proved
independently by Mulay \cite{Mu02} and DeMeyer and Schneider \cite{DS02}. Also, in \cite{Re01}, Redmond proved that if $\overline{\Gamma}(R)$ (the undirected zero-divisor graph of a non-commutative ring) contains a cycle, then ${\rm gr}(\overline{\Gamma}(R))\leq 4$. The following is our first main result in this section which shows that for a (not necessarily commutative) semigroup $S$, if $\overline{\Gamma}(S)$ contains a cycle, then $gr(\overline{\Gamma}(S))\leq 4$.

\begin{ttheo}\label{girth}
Let $S$ be a semigroup. If $\overline{\Gamma}(S)$ contains a cycle, then $gr(\overline{\Gamma}(S))\leq 4$.
\end{ttheo}

\begin{pproof}
Let $ a_{1}-a_{2}-\cdots -a_{n-1}-a_{n}-a_{1}$ be a cycle of shortest length
in $\overline{\Gamma}(S)$. Assume that
$gr(\overline{\Gamma}(S)) > 4$, i.e., assume $n\geq 5$. Note that $a_2a_{n-1} \neq 0$ and $a_{n-1}a_2 \neq 0$ ( as $n\geq 5$). If $a_2 a_{n-1}\not \in \{a_1, a_n\}$, then $a_1-a_2a_{n-1}-a_n-a_1$ is a cycle of length 3, yielding a contradiction. Also, if  $a_{n-1}a_2\not \in \{a_1, a_n\}$, then $a_1-a_{n-1}a_2-a_n-a_1$ is a cycle of length 3, yielding a contradiction. We have the following cases:

$Case ~1:$ $a_2a_{n-1}=a_1$ and $a_{n-1}a_2=a_n$. If $a_2a_3=0$, then $a_n a_3=(a_{n-1}a_2)a_3=0$. Therefore, $a_1-a_2-a_3-a_n-a_1$ is a cycle of length 4, yielding a contradiction. So, $a_3a_2=0$. Thus, $a_3a_1=a_3(a_2 a_{n-1})=0$. Therefore, $a_1-a_3-a_4-\cdots - a_{n-1}-a_n-a_1$ is a cycle of length $n-1$, yielding a contradiction.

$Case ~2:$ $a_2a_{n-1}=a_1$ and $a_{n-1}a_2=a_1$. If $a_2a_3=0$, then $a_1 a_3=(a_{n-1}a_2)a_3=0$. Therefore, $a_1-a_3-a_4-\cdots - a_{n-1}-a_n-a_1$ is a cycle of length $n-1$, yielding a contradiction. So, $a_3a_2=0$. Thus, $a_3a_1=a_3(a_2 a_{n-1})=0$. Therefore, $a_1-a_3-a_4-\cdots - a_{n-1}-a_n-a_1$ is a cycle of length $n-1$, yielding a contradiction.

$Case ~3:$ $a_2a_{n-1}=a_n$ and $a_{n-1}a_2=a_1$. If $a_2a_3=0$, then $a_1 a_3=(a_{n-1}a_2)a_3=0$. Therefore, $a_1-a_3-a_4-\cdots - a_{n-1}-a_n-a_1$ is a cycle of length $n-1$, yielding a contradiction.  So, $a_3a_2=0$. Thus, $a_3a_n=a_3(a_2 a_{n-1})=0$.  Therefore, $a_1-a_2-a_3-a_n-a_1$ is a cycle of length 4, yielding a contradiction.

$Case ~4:$ $a_2a_{n-1}=a_n$ and $a_{n-1}a_2=a_n$. If $a_2a_3=0$, then $a_n a_3=(a_{n-1}a_2)a_3=0$. If $a_3a_2=0$, then $a_3a_n=a_3(a_2 a_{n-1})=0$. Therefore, $a_1-a_2-a_3-a_n-a_1$ is a cycle of length $4$, yielding a contradiction.

Since in all cases we have found contradictions, we conclude that if $\overline{\Gamma}(S)$ contains a cycle, then $gr(\overline{\Gamma}(S)) \leq 4$.\hfill $\square$

\end{pproof}

\begin{ccoro}\label{diam3}
 Let $R$ be a ring. Then
$\overline{\Bbb{APOG}}(R)$ is a connected graph and
${\rm diam}(\overline{\Bbb{APOG}}(R)) \leq 3.$ Moreover,  If
$\overline{\Bbb{APOG}}(R)$ contains a cycle, then ${\rm gr}(\overline{\Bbb{APOG}}(R)) \leq 4$.
\end{ccoro}

\begin{pproof}
Note that  $\overline{\Bbb{APOG}}(R)$ is equal to $\overline{\Gamma}(\Bbb{IPO}(R))$. So by Theorem \ref{diam}, $\overline{\Bbb{APOG}}(R)$ is a connected graph and ${\rm diam}(\overline{\Bbb{APOG}}(R)) \leq 3$. Also, by Theorem \ref{girth},  if
$\overline{\Bbb{APOG}}(R)$ contains a cycle, then ${\rm gr}(\overline{\Bbb{APOG}}(R)) \leq 4$.\hfill $\square$
\end{pproof}
\\

For a not necessarily commutative ring $R$, we define a simple undirected graph $\overline{\Gamma}(R)$ with vertex set $D(R)^*$ (the set of all non-zero zero-divisors of $R$) in which two distinct
vertices $x$ and $y$ are adjacent if and only if either $xy = 0$ or $yx = 0$ (see  \cite{Re01}).  The Jacobson radical of $R$, denoted by $J(R)$, is equal to the intersection of all maximal right ideals of  $R$. It is well-known that $J(R)$ is also equal to the intersection of all maximal left ideals of $R$.    In our second main theorem in this section we characterize rings whose undirected annihilating-ideal graphs are complete graphs.

\begin{ttheo}
Let $R$ be a ring. Then $\overline{\Bbb{APOG}}(R)$ is a complete graph if and only if either $(D(R))^{2}=0$, or $R$ is a direct product of two division rings, or $R$ is a local ring with maximal ideal $\mathfrak{m}$ such that $\Bbb{IPO}(R)=\{0,\mathfrak{m},\mathfrak{m}^{2}, R\}$.
\end{ttheo}

\begin{pproof}
Assume that $\overline{\Bbb{APOG}}(R)$ is a complete graph. If $\overline{\Gamma}(R)$ is a complete graph, then by \cite[Theorem 5]{AM042},  either $R\cong \Bbb{Z}_2 \times \Bbb{Z}_2$ or $D(R)^2=\{0\}$. So the forward direction holds. Next assume that $\overline{\Gamma}(R)$ is not a complete graph. So there exist different vertices $x$ and $y$ of $\overline{\Gamma}(R)$ such that $x$ and $y$ are not adjacent. We have the following cases:

$Case~1$: $x\in A^{r}(R)$. Without loss of generality assume that $y\in A^r(R)$. If $Rx\neq Ry$, then since $\Bbb{APOG}(R)$ is a complete graph, we have $Rx$ is adjacent to $Ry$ in $\overline{\Bbb{APOG}}(R)$, so $x$ and $y$ are adjacent in $\overline{\Gamma}(R)$, yielding a contradiction. Thus $Rx=Ry$. Since $x\in A^{r}(R)$, there exists non-zero element $z\in D(R)$ such that $xz=0$. If $Rx\subseteq zR$, then $(Rx)^2=\{0\}$. So $(Rx)(Ry)=\{0\}$, and $x$ and $y$ are adjacent in $\overline{\Gamma}(R)$, yielding a cintradiction. Therefore, $Rx \nsubseteq zR$.  If there exists a left or right ideal $I$ of $R$ expect $zR$ such that $I\nsubseteq Rx$, then there exists nonzero element $s\in I\setminus Rx$.  Then $(Rs+Rx) (zR)=\{0\}$. Since $\overline{\Bbb{APOG}}(R)$ is a complete graph $Rx$ is adjacent to $(Rs+Rx)=\{0\}$. Thus $(Rx)^{2}=\{0\}$, and so  $x$ and $y$ are adjacent in $\overline{\Gamma}(R)$, yielding a contradiction. Therefore, $\{zR,Rx\}$ is the set of nonzero proper left or right ideals of $R$. Thus by Corollary \ref{Artin2}, $R$ is an Artinian ring. We have the following subcases:

$Subcase~1$: $zR\nsubseteq Rx$. Then $zR$ and $Rx$ are maximal ideals. If $zR$ or $Rx$ is not a two-sided ideal, then $zR=J(R)=Rx$, yielding a contradiction. Therefore, $Rx$ and $zR$ are two-sided ideals. Also, $Rx$ and $zR$ are minimal ideals and so  $Rx\cap zR=\{0\}$. Thus by Brauer's Lemma (see \cite[10.22]{La91}), $(Rx)^2=0$ or $Rx=Re$, where $e$ is a idempotent in $R$. If $(Rx)^2=\{0\}$, then $x$ is adjacent to $y$ in $\overline{\Gamma}(R)$, yielding a contradiction. So $Rx=Re$, where $e$ is an idempotent in $R$. Therefore, $R=eRe\oplus eR(1-e)\oplus (1-e)Re\oplus (1-e)R(1-e)$. Since $\{zR,Rx\}$ is the set of nonzero proper left or right ideals of $R$ and $Rx\cap zR=\{0\}$, we conclude that $Re=Rx=eR$ and $(1-e)R=zR=R(1-e)$. Therefore,  $(1-e)Re=(1-e)eR=\{0\}$ and $eR(1-e)=e(1-e)R=\{0\}$.
 So $R=eRe\oplus(1-e)R(1-e)$.   Since $R$ is an Artinian ring with two nonzero left or right ideals, we conclude that $eRe$ and $(1-e)R(1-e)$ are division rings.

$Subcase~2$: $zR\subseteq Rx$. Then $Rx=D(R)$. If $(Rx)^2=\{0\}$, then $x$ is adjacent to $y$ in $\overline{\Gamma}(R)$, yielding a contradiction. If $D(R)^2\neq 0$, then $D(R)^2=zR$. Therefore, $R$ is a local ring with maximal ideal $\mathfrak{m}$ such that $\Bbb{IPO}(R)=\{0,\mathfrak{m},\mathfrak{m}^{2}, R\}$.

In summary, we obtain that either $R$ is a direct product of two division rings, or $R$ is a local ring with maximal ideal $\mathfrak{m}$ such that $\Bbb{IPO}(R)=\{0,\mathfrak{m},\mathfrak{m}^{2}, R\}$. Thus the forward direction holds.

$Case~2$:  $x\in A^{l}(R)$. Similar to Case 1, we  conclude that either $R$ is a direct product of two division rings, or $R$ is a local ring with maximal ideal $\mathfrak{m}$ such that $\Bbb{IPO}(R)=\{0,\mathfrak{m},\mathfrak{m}^{2}, R\}$. So the forward direction holds.\\

The converse is obvious. \hfill $\square$

\end{pproof}


\section{Undirected Annihilating-Ideal Graphs for Matrix Rings Over Commutative Rings }

In this section we investigate the undirected annihilating-ideal graphs of matrix rings over  commutative rings. By Theorem \ref{diam3}, ${\rm diam}(\overline{\Bbb{APOG}}(R)) \leq 3$ for any ring $R$. In Proposition~\ref{mdiam} we show that ${\rm diam}((\overline{\Bbb{APOG}}(M_{n}(R)))\geq 2$ where $n\geq 2$. A natural question is whether or not  ${\rm diam}(\overline{\Bbb{APOG}}(M_{n}(R))\geq {\rm diam}(\overline{\Bbb{APOG}}(R))$.  We show that the answer to this question is affirmative.

\begin{ppro}\label{mdiam}
 Let $R$ be a commutative ring. Then ${\rm diam}(\overline{\Bbb{APOG}}(M_{n}(R))\geq 2$ where $n\geq 2$.
\end{ppro}

\begin{pproof}
Let $$A=(M_{n}(R)\begin{bmatrix}
1 & 0 & 0 & \cdots & 0\\
0 & 0 & 0 & \cdots & 0\\
\vdots & \vdots & \vdots & \cdots & \vdots \\
0 & 0 & 0 & \cdots & 0
\end{bmatrix}) ~{\rm and}~ B=(\begin{bmatrix}
1 & 0 & 0 & \cdots & 0\\
0 & 0 & 0 & \cdots & 0\\
\vdots & \vdots & \vdots & \cdots & \vdots \\
0 & 0 & 0 & \cdots & 0
\end{bmatrix}M_{n}(R)).$$
Since
$$A(\begin{bmatrix}
0 & 0 & 0 & \cdots & 0\\
1 & 0 & 0 & \cdots & 0\\
\vdots & \vdots & \vdots & \cdots & \vdots \\
0 & 0 & 0 & \cdots & 0
\end{bmatrix}M_{n}(R))=0 ~{\rm and} ~(M_{n}(R)\begin{bmatrix}
0 & 0 & 0 & \cdots & 0\\
1 & 0 & 0 & \cdots & 0\\
\vdots & \vdots & \vdots & \cdots & \vdots \\
0 & 0 & 0 & \cdots & 0
\end{bmatrix})B=0, $$
we conclude that $A$ and $B$ are vertices in $(\overline{\Bbb{APOG}}(M_{n}(R))$.
Note that
 $$\begin{bmatrix}
1 & 0 & 0 & \cdots & 0\\
0 & 0 & 0 & \cdots & 0\\
\vdots & \vdots & \vdots & \cdots & \vdots \\
0 & 0 & 0 & \cdots & 0
\end{bmatrix}^{2}\neq 0 ~{\rm and}~
\begin{bmatrix} 1 & 0 & 0 & \cdots & 0\\
0 & 0 & 0 & \cdots & 0\\
\vdots & \vdots & \vdots & \cdots & \vdots \\
0 & 0 & 0 & \cdots & 0
\end{bmatrix}\in A\cap B,$$ so $AB\neq 0$. Therefore, ${\rm diam}(\overline{\Bbb{APOG}}(M_{n}(R))\geq 2$.\hfill $\square$
\end{pproof}

\begin{ttheo}
 Let $R$ be a commutative ring. Then ${\rm diam}(\overline{\Bbb{APOG}}(M_{n}(R))\geq {\rm diam}(\Bbb{AG}(R))={\rm diam}(\overline{\Bbb{APOG}}(R))$.
\end{ttheo}

\begin{pproof}
By \cite[Theorem 2.1]{BRI11}, ${\rm diam} (\Bbb{AG}(R))\leq 3$.

$Case$ 1: $diam(\Bbb{AG}(R))\leq 2$. By Proposition \ref{mdiam}, ${\rm diam}(\overline{\Bbb{APOG}}(M_{n}(R))\geq 2$. Thus ${\rm diam} (\overline{\Bbb{APOG}}(M_{n}(R))\geq {\rm diam}(\Bbb{AG}(R))$.

$Case$ 2: ${\rm diam}(\Bbb{AG}(R))=3$. Then there exist vertices  $I,J,K$, and $L$ of $\Bbb{AG}(R)$ such that $I-K-L-J$ is a shortest path between $I$ and $J$. So $d(I, J)=3$.  Since $I$ and $J$ are vertices of $\Bbb{AG}(R)$, $M_n(I)$ and $M_n(J)$ are vertices of  $\overline{\Bbb{APOG}}(M_n(R))$. Suppose that  ${\rm diam}(\overline{\Bbb{APOG}}(M_{n}(R))=2$. So we can assume that there exists $\alpha=[a_{ij}]\in M_{n}(R)$ such that $M_{n}(I)\alpha=\alpha M_{n}(J)=0$. Without loss of generality, we may  assume that $a_{11}\neq 0$. For every $a\in I$,
$$
\begin{bmatrix}
a & 0 & 0 & \cdots & 0\\
0 & 0 & 0 & \cdots & 0\\
\vdots & \vdots & \vdots & \cdots & \vdots \\
0 & 0 & 0 & \cdots & 0
\end{bmatrix}A=0,$$ so $aa_{11}=0$. Therefore $I(a_{11}R)=0$. For every $b\in J$,
$$A\begin{bmatrix}
b & 0 & 0 & \cdots & 0\\
0 & 0 & 0 & \cdots & 0\\
\vdots & \vdots & \vdots & \cdots & \vdots \\
0 & 0 & 0 & \cdots & 0
\end{bmatrix}=0.$$ Therefore  $(a_{11}R)J=0$. Thus $I-(a_{11}R)-J$ is a path of length 2 in $\Bbb{AG}(R)$, and so $d(I,J) \leq 2$, yielding a contradiction. Therefore, ${\rm diam}(\overline{\Bbb{APOG}}(M_{n}(R))=3$ and we are done. \hfill $\square$
\end{pproof}
\\

It was shown in Corollary \ref{diam3} that ${\rm gr}(\overline{\Bbb{APOG}}(R)) \leq 4$. We now show that ${\rm gr}(\overline{\Bbb{APOG}}(M_{n}(R)))=3$ where $n\geq 2$.

\begin{ppro}
 Let $R$ be a commutative ring. Then ${\rm gr}(\overline{\Bbb{APOG}}(M_{n}(R))=3$ where $n\geq 2$.
\end{ppro}

\begin{pproof}
Let
$$
A=\begin{bmatrix}
1 & 1 & 0 & \cdots & 0\\
0 & 0 & 0 & \cdots & 0\\
\vdots & \vdots & \vdots & \cdots & \vdots \\
0 & 0 & 0 & \cdots & 0
\end{bmatrix}
,
B=\begin{bmatrix}
1 & -1 & 0 & \cdots & 0\\
-1 & 1 & 0 & \cdots & 0\\
\vdots & \vdots & \vdots & \cdots & \vdots \\
0 & 0 & 0 & \cdots & 0
\end{bmatrix}
,$$ and
$$
C=\begin{bmatrix}
0 & 1 & 0 & \cdots & 0\\
0 & 1 & 0 & \cdots & 0\\
\vdots & \vdots & \vdots & \cdots & \vdots \\
0 & 0 & 0 & \cdots & 0
\end{bmatrix}.$$
Then $(AM_{n}(R)A) - (BM_{n}(R)B) - (CM_{n}(R)C)$ is a cycle in $(\overline{\Bbb{APOG}}(M_{n}(R))$, so ${\rm gr}(\overline{\Bbb{APOG}}(M_{n}(R))=3$.\hfill $\square$
\end{pproof}

Farid Aliniaeifard, Department of Mathematics, Brock University, St. Catharines, Ontario, Canada L2S 3A1. Tel: (905) 688-5550 and Fax:(905) 378-5713. E-mail address: fa11da@brocku.ca
\\

Mahmood Behboodi, Department of Mathematical of Sciences, Isfahan University of Technology, Isfahan, Iran 84156-8311. Tel : (+98)(311) 391-3612 and Fax : (+98)(311) 391-3602. E-mail address: mbehbood@cc.iut.ac.ir
\\

Yuanlin Li, Department of Mathematics, Brock University, St. Catharines, Ontario, Canada L2S 3A1. Tel: (905) 688-5550 ext. 4626  and Fax:(905) 378-5713. E-mail address: yli@brocku.ca

\end{document}